\documentclass{amsart}
\usepackage{amsmath,amssymb,amsthm,mathrsfs,amsrefs,microtype}
\usepackage{booktabs}
\usepackage{comment}

\newtheorem{theorem}{Theorem}
\newtheorem{lemma}[theorem]{Lemma}
\newtheorem{prop}[theorem]{Proposition}
\newtheorem{corollary}[theorem]{Corollary}
\theoremstyle{remark}
\newtheorem{example}[theorem]{Example}
\newtheorem{remark}[theorem]{Remark}
\newcommand{\ZZ}{\mathbb{Z}}

\newcommand{\PP}{\mathbb{P}}
\renewcommand{\AA}{\mathbb{A}}

\newcommand{\dsum}{D}
\newcommand{\fk}{K}

\DeclareMathOperator{\lt}{lt}
\title{Polynomial systems admitting a simultaneous solution}

\author[A. Conner]{Austin Conner}
\address[A. Conner]{Department of Mathematics and Statistics, University of Konstanz, 
Germany, and Department of Mathematics, Harvard University, USA}
\email{aconner.vu@gmail.com}

\author[M.~Micha{\l}ek]{Mateusz Micha{\l}ek}
\address[M.~Micha{\l}ek]{Department of Mathematics and Statistics, University of Konstanz,
Germany}
\email{mateusz.michalek@uni-konstanz.de}

\author[M.~Schindler]{Michael Schindler}
\address[M.~Schindler]{CNRS~UMR7083, ESPCI~Paris, Université~PSL, %10~rue~Vauquelin, 75005 Paris, 
France}
\email{michael.schindler@espci.fr}

\author[B.~Szendr\H oi]{Bal\' azs Szendr\H oi}
\address[B.~Szendr\H oi]{Faculty of Mathematics, University of Vienna, %Oscar-Morgenstein-Platz~1, 1090 Vienna, 
Austria}
\email{balazs.szendroi@univie.ac.at}

\date{January 2024}
\begin{document}

\begin{abstract}
We provide a description of a complete set of generators for the ideal that serves as the resultant ideal for $n$ univariate polynomials of degree $d$. Our generators arise as maximal minors of a set of cascading matrices formed from the coefficients of the polynomials, generalising the classical Sylvester resultant of two polynomials. 
\end{abstract}
\maketitle
\section{Introduction}
Fix integers $n>1$ and $d>1$. Consider a system 
\begin{equation}\label{eq:system}
  f_i(x) := a_{i,0} x^d + a_{i,1} x^{d-1} + \cdots + a_{i,d} = 0 \quad \text{ \ $1\le i\le n$}
\end{equation}
of $n$ univariate polynomials of degree $d$ in a variable $x$ over an algebraically closed field $\fk$. 
A natural question arises: 
\emph{when do the polynomials~$f_i$ have a common root?}

By eliminating the variable $x$ from the ideal $\langle f_i(x)\rangle$, we obtain a radical ideal 
\[I_{d,n}\lhd  \fk[a_{i,j}]_{1\leq i\leq n, 0\leq j\leq d}\]
in the polynomial ring of coefficients, which serves as a resultant for the set of polynomials $\{f_i(x)\}$ in the following sense. 
\begin{enumerate}
     \item[(a)] If the polynomials $f_i(x)$ have a common root, then the coefficients $a_{i,j}$ belong to the variety $V(I_{d,n})$.
     \item[(b)] If the coefficients $a_{i,j}$ belong to the variety $V(I_{d,n})$, then either the polynomials $f_i(x)$ have a common root, or for all $i$ we have $a_{i,0}=0$.
\end{enumerate}
%The condition that ``the polynomials $f_i(x)$ have a common root, or for all $i$ we have $a_{i,0}=0$'' 
The conclusion in (b) simply means that the associated binary forms 
\[g_i(x,y):=a_{i,0} x^d + a_{i,1} x^{d-1}y + \cdots + a_{i,d}y^d\]
have a common root in $\PP^1_{[x:y]}$. This is a Zariski closed condition in the projective space defined by the coefficients $a_{i,j}$, and is the closure of the condition that the polynomials $f_i(x)$ have a common root in $\AA^1_x$. Thus the ideal $I_{d,n}$ is the fundamental object providing the answer to our basic question; we will call it the {\em resultant ideal} of the polynomial system~\eqref{eq:system}, following the terminology in \cite{jouanolou}.

In this paper, we give a description of a complete set of generators of the ideal~$I_{d,n}$. Aspects of this very natural and classical problem have been investigated since the 19th century. The best known is the case $n=2$ of two equations. As observed by Sylvester, the ideal $I_{d,2}$ is principal, generated by the resultant polynomial ${\rm Res}(f_1,f_2)\in \fk[a_{i,j}]$, the determinant of a matrix known nowadays as the Sylvester matrix. For the next case $d=2, n=3$, the ideal $I_{2,3}$ is easily computed (at least by computer algebra), and was studied earlier in~\cite{tropical}*{Ex.\;5.6, Ex.\;6.6}. For small, fixed $n>2$ and $d$, one can still give an explicit set of generators for~$I_{d,n}$ via elimination. However, this quickly becomes impossible, and the answer intractable. 

The ideal~$I_{d,n}$, as well as the variety~$V(I_{d,n})$, the locus of forms that have a common root, have also been studied from a theoretical point of view. A classical reference is van der Waerden~\cite[\S 130]{vdWaerden}, where some of the properties of $I_{d,n}$ are described. However, in all modern editions, the short arguments involving this ideal are nonconstructive, only appearing as corollaries of the Nullstellensatz; older versions made more explicit use of the theory of resultants. 

A question closely related to ours was already answered a long time ago by Orsinger~\cite[Satz 7]{orsinger}, though this does not appear to be generally known~\cite{MOQ}, even for the case of quadratic polynomials. This is the set-theoretic question of giving polynomial conditions for the coefficients $a_{ij}$ that guarantee the existence of a common root. Orsinger's result was rediscovered by Kaki\'{e}~\cite{kakie} and also in greater generality by Jouanolou~\cite[Section 3.3.7]{jouanolou}. The question of finding the minimal number of polynomial conditions ensuring a common root was investigated by Lyubeznik~\cite{Lub}. However, these results do not approach the problem in an ideal-theoretic sense: the polynomials they provide do not generate the resultant ideal~$I_{d,n}$, for simple degree reasons. What they generate instead is a non-radical ideal with radical $I_{d,n}$. This phenomenon already occurs for quadratic polynomials, where for $n>2$ a natural set of generators for $I_{2,n}$ include classical resultant quartics, some further degree-$4$ relations already contained in~\cites{orsinger, kakie}, as well as cubic relations. Jouanolou~\cite{jouanolou} discusses many further properties of the ideal $I_{d,n}$.  However, to our knowledge, the explicit description of a generating set for the ideal $I_{d,n}$ was not known before our work. In particular, as observed by Jan Stevens in private communication, our results imply that the determinantal equations considered by Orsinger and Kaki\'e define the correct projective scheme, although they do not generate the correct ideal. See Remark~\ref{rem_sylv}, Corollary~\ref{cor_kakie} and Proposition~\ref{prop:sat} for further discussion.  

In this article we provide a description of the resultant ideal $I_{d,n}$ in the following sense:  we provide
\begin{enumerate}
    \item[(a)] a list of generators, in determinantal form, for $I_{d,n}$;
    \item[(b)] a Gr\"obner basis for $I_{d,n}$;
    \item[(c)] the degree and the dimension of the variety $X_{d,n} := V(I_{d,n}) \subset \PP^{n(d+1)-1}$. 
\end{enumerate}

Here (c) is straightforward, using a natural resolution of singularities of $X_{d,n}$ via vector bundles, the subject of our Section~\ref{sect_deg}. At the start of Section~\ref{sect_main}, we provide a set of determinantal elements in the resultant ideal. Our strategy to solve (a)-(b), and in particular to prove our main result Theorem~\ref{thm:minors}, is as follows. First, we will pick a term order on the polynomial ring of coefficients, and a subset $G\subset I_{d,n}$, which will eventually be shown to be a Gr\"obner basis. We will show that the leading terms of $G$ are square-free and that the variety defined by the corresponding initial ideal is the union of $\deg X_{d,n}$ coordinate
subspaces and of dimension equal to $\dim X_{d,n}$. As we will argue, these facts establish that $G$ is a Gr\"obner basis of $I_{d,n}$. We conclude the paper in Section~\ref{sect_conc} with final remarks, in particular recovering the set-theoretic description as a special case. 

In Section~\ref{sect_deg}, we will be working in a more general setting, where the polynomials in the system~\eqref{eq:system} can have different degrees. However, from Section~\ref{sect_main}, we focus on the case of polynomials of equal degree $d$. We hope to return to the more general case in later work. 

We will be assuming familiarity with ideals, Gr\"obner bases, term orders, elimination theory and the Nullstellensatz, as presented in~\cite{MS}*{Chapters 1--4 and~6}.  We also rely on basic intersection theory, referring the interested reader to~\cite{3264}.
%Indeed, if it were not, by the Nullstellensatz and the fact that $\lt G$ is radical, $V(\lt I_{d,n})$ must be strictly contained in $V(\lt G)$, and in particular must have either strictly smaller dimension than $\dim X_{d,n}$ or have strictly smaller degree than $\deg X_{d,n}$. However, $X_{d,n}$ degenerates to $V(\lt I_{d,n})$, and in particular they have the same dimension and degree, which gives us our contradiction. IS THIS REALLY NEEDED HERE? PERHAPS TOO MUCH DETAIL FOR THE INTRO. B

\section*{Acknowledgements} A.C. was~supported by NSF grant 2002149 and DFG grant 467575307. M.M. was supported by DFG grant 467575307. We would like to thank Bernd Sturmfels for many inspiring talks on the topic, and for some influential comments on early versions of our result. We are very grateful to Jan Stevens for pointing out several important classical results and sources in the field, as well as Proposition~\ref{prop:sat}. We thank Elke Neuhaus for remarks about the first version of the article. The last-named author would also like to thank Jerzy Weyman for a conversation on this subject. 

\section{Dimension and degree}\label{sect_deg}
We start by computing the dimension and degree of the projective variety
\[X_{d,n} =V(I_{d,n}) \subset \PP^{n(d+1)-1},\] 
the vanishing locus of the resultant ideal, defined in the Introduction. Note that different proofs of these results were given in~\cite[Lem.~6.3, Prop.~6.5]{tropical}.

Let us slightly generalize the setting: consider $n$ bivariate, homogeneous forms
\[g_i(x,y):=a_{i,0} x^{d_i} + a_{i,1} x^{d_i-1}y + \cdots + a_{i,d_i}y^{d_i}, \ \ i=1,\dots,n\]
of possibly distinct degrees $d_1,\dots,d_n$. Denoting $\dsum:=\sum_{i=1}^n d_i$, the space of such forms is parameterized by the affine space $\fk^{n+\dsum}$ of coefficients $a_{i,j}$.   
Let $X_{d_1,\dots,d_n}\subset \PP^{n-1+\dsum}$ be the locus inside the projective space of coefficients corresponding to those $n$-tuples of forms that have a common root in $\PP^1$. 
\begin{prop}
The set $X_{d_1,\dots,d_n}\subset \PP^{n-1+\dsum}$ is an irreducible projective variety of dimension $\dsum$ and degree $\dsum$.
\end{prop}
\begin{proof}
Consider the projective line $\PP^1$ with coordinates $x,y$. Inside $\PP^1\times \PP^{n-1+\dsum}$, each binary form $g_i(x,y)$ defines a codimension one subvariety $B_i$. Projecting $B_i$ to $\PP^1$ makes $B_i$ into a projective bundle of rank $n-2+\dsum$ over $\PP^1$, a codimension one subbundle of the trivial bundle. 

We now prove that all $B_i$'s intersect transversally. Indeed, as bundles are locally trivial, they intersect transversally if and only if they intersect transversally on every fiber. However, for fixed $[x:y]$ each $g_i$ becomes a linear equation in a \emph{distinct} set of variables. In particular, the linear equations are independent and thus the intersection is transversal. It follows that the variety $Y_{d_1,\dots,d_n}:=\bigcap_{i=1}^n B_i$ is also a projective bundle over $\PP^1$ of rank $(n-1+\dsum)-n=\dsum-1$. Hence, $\dim Y_{d_1,\dots,d_n}=\dsum$. 

Consider the projection $\PP^1\times \PP^{n-1+\dsum}\rightarrow \PP^{n-1+\dsum}$. We claim that the image of $Y_{d_1,\dots,d_n}$ is precisely $X_{d_1,\dots,d_n}$. Indeed, a point $([x:y],[a_{i,j}])$ belongs to $Y_{d_1,\dots,d_n}$ if and only if $[x:y]$ is a common root of the binary forms $g_i(x,y)$. In particular, $X_{d_1,\dots,d_n}$ is an irreducible variety. Further, 
% we claim that 
the resulting map \[\pi\colon Y_{d_1,\dots,d_n}\to X_{d_1,\dots,d_n}\] is birational~\cite[Prop.3.3.1]{jouanolou}. Indeed, the general fiber is a singleton, as, for general $g_i(x,y)$ having a common root, this root is unique. 
%\footnote{It is easy to exhibit one case when the equations have exactly one common root and notice that this remains true in a neighbourhood.}
Thus, $\dim X_{d_1,\dots,d_n}=\dim Y_{d_1,\dots,d_n}=D$.

As a side remark, we note that $\pi$ is not an isomorphism, as some systems have several common solutions. In particular, $X_{d_1,\dots,d_n}$ in general is singular, while $Y_{d_1,\dots,d_n}$ is always smooth, with $\pi$ a resolution of singularities of $X_{d_1,\dots,d_n}$.

Recall that the degree of $X_{d_1,\dots,d_n}$ is the number of points we obtain after intersecting it with $\dsum$ general hyperplanes in $\PP^{n-1+\dsum}$.
%We note that $\deg X_{d_1,\dots,d_n}$ equals the last entry of the multidegree of $Y_{d_1,\dots,d_n}$. 
Pulling back hyperplanes of $\PP^{n-1+\dsum}$ by the projection map, we obtain divisors on $\PP^1\times \PP^{n-1+\dsum}$ that belong to a base-point-free linear system ${\mathcal O}_{\PP^1\times \PP^{n-1+\dsum}}(H_2)$, the pullback of the hyperplane system on the second factor. Intersecting $Y_{d,n}$ with $\dsum$ general divisors from this linear system, by Bertini's theorem we obtain a finite number $k$ of reduced points of $Y$ that are general in the sense that they belong to the open complement $Y_{d_1,\dots,d_n}\setminus{\rm Exc}(\pi)$ of the exceptional locus of $\pi$. Let us note that as the intersection points belong to the locus where $Y_{d_1,\dots,d_n}$ and $X_{d_1,\dots,d_n}$ are isomorphic, to know that we obtain reduced points, it is enough to apply Bertini's theorem for the complete linear system of hyperplanes in the projective space, which holds in arbitrary characteristic of the field. 
It follows that $k=\deg X_{d_1,\dots,d_n}$.

It remains to compute the number of points we obtain by intersecting $Y_{d_1,\dots,d_n}$ with $\dsum$ divisors of the linear system ${\mathcal O}_{\PP^1\times \PP^{n-1+\dsum}}(H_2)$. Recall that the Chow ring of $ \PP^1\times \PP^{n-1+\dsum}$ is $R=\ZZ[H_1,H_2]/(H_1^2, H_2^{n+\dsum})$, with $H_i$ the hyperplane class pulled back from each factor, the class of the point being the top nonzero intersection $H_1H_2^{n-1+\dsum}$. 
%Here, $H_2$ is the divisor as above and $H_1$ is (a point in $\PP^1$) times $\PP^{n-1+\dsum}$. 

Each divisor $B_i$ is of degree $d_i$ in $x,y$ and degree $1$ in the coefficient variables $a_{i,j}$. Its class is thus $d_iH_1+H_2\in R$. As we proved that $Y_{d_1,\dots,d_n}$ is a transversal intersection of the hypersurfaces $B_i$, its class is the product $\prod_{i=1}^n(d_iH_1+H_2)\in R$. It remains to compute the intersection with $\dsum$ divisors of class $H_2$, which are general, hence transversal by Bertini's theorem, to deduce
\[H_2^{\dsum}\prod_{i=1}^n(d_iH_1+H_2)=\dsum \cdot H_1H_2^{n-1+\dsum}\in R,\]
and thus $k=D$.
\end{proof}

\begin{corollary}\label{cor:degdimdn} The projective variety $X_{d,n} \subset \PP^{n(d+1)-1}$ is irreducible of dimension $\dim X_{d,n} = nd$ and degree $\deg X_{d,n} = nd$.
\end{corollary}

\begin{remark} As argued above, $Y_{d,n}\subset\PP^{n(d+1)-1}\times\PP^1$ is a complete intersection, and hence its ideal can be resolved by the Koszul complex~\cite[(1.7)]{jouanolou}. Pushing forward this resolution along the map $\pi$, together with a standard computation\footnote{We would like to thank Jerzy Weyman for explaining this.}, shows that $X_{d,n}$ is not normal. A full resolution of the ideal of the embedding $X_{d,n} \subset \PP^{n(d+1)-1}$ is studied in~\cite[Section 4]{jouanolou}. 
In that setting, a certain module $\tilde \Gamma$ surjects onto the ideal of $X_{d,n}$. Thus knowing the generators of~$\tilde \Gamma$ would provide generators of the ideal 
of~$X_{d,n}$.
The module $\tilde \Gamma$ may be realized as a kernel of a map in a certain degree in the dual of a Koszul complex (see \cite[Remarque 4.5.3]{jouanolou}). As Jouanolou writes, thanks to the results of Hermann and Hilbert, this allows one in principle to obtain information about the generators of~$\tilde \Gamma$ and hence about the generators of the ideal of~$X_{d,n}$. On the other hand, computing a set of generators of this kernel is a very hard task, although algorithmically doable.   
The relationship between Jouanolou's resolution and our main result Theorem~\ref{thm:minors} below deserves further study.  
\end{remark}

\section{Determinantal equations and the main result}\label{sect_main}

By definition,~$I_{d,n}$ is the radical ideal defining the irreducible variety~$X_{d,n} \subset \PP^{n(d+1)-1}$, and hence prime (see also~\cite[3.3.7]{jouanolou}). Our next step is to construct determinantal equations in $I_{d,n}$. For $1\leq k \leq d$, define the $nk \times (d+k)$ matrix 
% these matrices have a natural linear map interpretetion, could mention this
\[ M_k = 
  \begin{bmatrix}
  a_{1,0} & \cdots & a_{1,d} & & & \\
          & \vdots &         & & &  &   \\
  a_{n,0} & \cdots & a_{n,d} & & & & \\
    & a_{1,0} & \cdots & a_{1,d}  & & \\
    &         & \vdots &         & & \\
    & a_{n,0} & \cdots & a_{n,d} & & \\
    &         &  & \ddots & & &  \\
    & & & a_{1,0} & \cdots & a_{1,d} \\
    &        &           & \vdots & \\
    & & & a_{n,0} & \cdots & a_{n,d} \\
\end{bmatrix},
\]
where each rectangular matrix is shifted to the right by one step corresponding to the rectangle above it, there
are a total of $k$ rectangles, and the unspecified entries are all zero.
%MM:We could piont out here Sylvester, or maybe even earlier in intro
If $x$ is a common solution to the system \eqref{eq:system}, then $[x^{d+k-1}, x^{d+k-2},
\ldots, 1]^t$ lies in the kernel of $M_k$, hence all $(d+k) \times (d+k)$ minors of
$M_k$ lie in $I_{d,n}$. Alternatively, writing $P_s$ for the vector space of univariate
polynomials of degree at most $s$, $M_k^t$ can be interpreted as the map $
(P_k)^{\times n} \to P_{d+k} $, $ (h_1,\ldots,h_n) \mapsto f_1h_1+\cdots +
f_nh_n$. This map is rank deficient when $x$ is a common root of the polynomials
$f_i$, as then it is a common root of the entire image. Thus, the Fitting ideal corresponding to rank $(d+k)$ is contained in $I_{d,n}$.

The following is our main result. 

\begin{theorem}\label{thm:minors}
  The resultant ideal $I_{d,n}\lhd  \fk[a_{i,j}]_{1\leq i\leq n, 0\leq j\leq d}$ is generated by all $(d+k) \times (d+k)$ minors of $M_k$ for $1\le k\le d$.
\end{theorem}

\begin{remark}\label{rem_sylv}
In the above theorem, if $n$ is small with respect to $d$, there may be no appropriately
sized minors of $M_k$ for small $k$. For instance, the theorem asserts that in the case $n=2$,
$I_{d,2}$ is generated by a single $2d\times 2d$ determinant, the classical Sylvester resultant ${\rm Res}(f_1,f_2)$.

More generally, it is easy to see that for the largest value $k=d$, some of the $2d \times 2d$
minors of $M_d$ are just the pairwise resultants ${\rm Res}(f_i,f_j)$ of the original polynomials. The vanishing of the entire set of $2d \times 2d$ minors of $M_d$ is precisely the set-theoretic condition for the existence of a common root found by Kaki\'{e}~\cite{kakie}; we will recover this result below in Corollary~\ref{cor_kakie}. These degree~$2d$ polynomials alone clearly cannot generate the ideal $I_{d,n}$, as the minors of the smaller $M_k$ are of lower degree $d+k$.

At the other extreme $k=1$, the condition on the rank of $M_1$ is simply the condition that if $d+1$ univariate polynomials of degree $d$ share a root, then these polynomials are linearly dependent; this is easy to see directly. However, it is also immediate (for example for dimension reasons) that the set of equations $\mathrm{rk}\, M_1<d+1$ is not sufficient, even set-theoretically, to force a common root. 
\end{remark}

We will use Gr\"obner basis techniques to prove Theorem~\ref{thm:minors}. The first step is to establish a term order on the polynomial ring $\fk[a_{i,j}]_{1\leq i\leq n, 0\leq j\leq d}$, which is achieved in the next proposition. 

\begin{prop}\label{prop:ltdiag}
  There is a term order on $\fk[a_{i,j}]_{1\leq i\leq n, 0\leq j\leq d}$
  with the property that the leading monomial of any minor of $M_k$ is the product of its diagonal elements.
\end{prop}
\begin{remark}
  Observe that the content of this proposition is sensitive to the order of the rows of $M_k$, even as the set of minors up to sign is not. If the order of the
  rows of $M_k$ is permuted, the meaning of the diagonals of minors will change and the claim may no longer hold.
\end{remark}
\begin{remark}
  The proposition implies that a minor of $M_k$ is not identically zero exactly when its
  diagonal contains no zeros. The fact that a zero on the diagonal implies that
  the minor is zero can also be seen directly.
\end{remark}
\begin{proof}[Proof of Proposition \ref{prop:ltdiag}]
Fix an increasing sequence of positive numbers
\begin{multline*}
    x_{n,1}<x_{n-1,1}<\dots <x_{1,1}<x_{n,2}<x_{n-1,2}<\dots<x_{1,2}<x_{n,3}<\dots\\
    \dots<x_{n,d}<x_{n-1,d}<\dots<x_{1,d}.
\end{multline*}
Assign weight one to each $a_{k,d}$ for $k=1,\dots,n$. Inductively, from $l=d$ and going down to $l=0$, assign weight $w_{k,l}$ to each $a_{k,l}$ so that $w_{k,l}-w_{k,l+1}=x_{k,l+1}$. We claim that any term order compatible with the given weights will choose the diagonal as a leading term for any minor.

For contradiction assume this is not the case and fix a minor for which the leading term is not the diagonal. Let $X=(y_{i,j})$ be the corresponding submatrix. If the leading term does not correspond to the diagonal then it must be divisible by the product $y_{i,j}y_{p,q}$ so that $i<p$ and $q<j$. We claim that replacing this term by $y_{i,q}y_{p,j}$ would increase the weight, which gives the contradiction.

Say $y_{i,j}=a_{i',j'}$ and $y_{p,q}=a_{p',q'}$. Then there is a $c$ such that $y_{i,q}=a_{i',j'-c}$ and $y_{p,j}=a_{p',q'+c}$. In particular, we note that these are nonzero, as $q'<q'+c,j'-c\leq j'$. It remains to observe that the difference of weights $w_{i',j'-c}-w_{i',j'}$ is greater than the difference of weights $w_{p',q'}-w_{p',q'+c}$, which follows from the choice of $x_{i,j}$'s.
%MM: maybe I should write more details for the last sentence? Here a matrix picture really helps
 \begin{comment}
  We prove the claim by exhibiting an explicit term order with the property.
  Define $\prod_{i,j} a_{i,j}^{A_{i,j}} > \prod_{i,j} a_{i,j}^{B_{i,j}}$ if and only if
  \begin{enumerate}
    \item $\sum_{i,j} A_{i,j} > \sum_{i,j} B_{i,j}$, or 
    \item $\sum_{i,j} A_{i,j} = \sum_{i,j} B_{i,j}$ and there is a $j_0$ so
      that $\sum_i A_{i,j_0} < \sum_i B_{i,j_0}$ and $\sum_i A_{i,j} = \sum_i
      B_{i,j}$ when $j > j_0$, or
    \item $\sum_i A_{i,j} = \sum_i B_{i,j}$ for all $j$, and there is an $i_0,j_0$ so that $A_{i_0,j_0} > B_{i_0,j_0}$ and $A_{i,j} = B_{i,j}$ when either $i < i_0$ or $i=i_0$ and $j<j_0$.
\end{enumerate}

An alternative description of the order is the following. Consider indeterminates $b_j$ and
$c_{i,j}$, and endow monomials in $b_j$ with the degree reverse lexicographic
order and monomials in $c_{i,j}$ with the lexicographic order
with variables increasing down each column, column by column left to right.
Given a monomial $a^A$, write $b^A$ for $a^A$ where $b_j$ is substituted
for $a_{i,j}$ and write $c^A$ for $a^A$ where $c_{i,j}$ is
substituted for $a_{i,j}$. Then $a^A > a^B$ if and only if $b^A > b^B$ or $b^A
= b^B$ and $c^A > c^B$.

TODO Need to prove the claim.
\end{comment}
\end{proof}
Choosing a $(d+k) \times (d+k)$ minor of $M_k$ is the same as choosing a subset of
the rows of size $d+k$. Rows of $M_k$ are naturally indexed by pairs $(i,j)$, where
$1\le i\le k$ and $1\le j \le n$, so we may identify such minors with
their lexicographically ordered list of pairs $( (i_1,j_1), \ldots, (i_{d+k},
j_{d+k}))$. Here, the ordering corresponds to taking rows of $M_k$ from top to
bottom.

\begin{corollary}
  Write $a_{i,j} = 0$ if $j < 0$ or $j > d$.
  The leading monomial of the minor $( (i_1,j_1), \ldots, (i_{d+k}, j_{d+k}))$ is
  $\prod_{s=1}^{d+k} a_{j_s,s-i_s}$.
\end{corollary}
For fixed minor, for each $s$, write $(u_s,v_s) := (j_s,s-i_s)$. When $s > 1$, we have that either $v_s \le
v_{s-1}$ or ($v_s = v_{s-1} + 1$ and $u_s > u_{s-1}$). These correspond to the
cases that $i_s > i_{s-1}$ and $i_s = i_{s-1}$, respectively. Restrict attention
now to nonzero minors. These are exactly those for which each $(u_s,v_s)$ is contained
within the $n\times (d+1)$ lattice. We always have $v_1 = 1-i_1\le 0$ and $v_{d+k} =
d+k-i_{d+k} \ge d$, so
for a nonzero minor in particular equality holds for both. Call a walk $(u_1,v_1),\dots,(u_{d+k},v_{d+k})$ through the
lattice satisfying the conditions
\begin{enumerate}
\item $v_{s+1}\le v_s$ or both $v_{s+1} = v_s + 1$ and $u_{s+1} > u_s$;
\item $v_1=0$ and $v_{d+k}=d$
\end{enumerate}
 a \emph{minor walk}. Every minor walk arises
from an actual minor: the only thing to check is that $i_s$ satisfies
$1\le i_s \le k$. These are the conditions $s-k \le v_s\le s-1$. The 
upper bound follows from $v_1 = 0$ and $v_{s+1} \le v_s+1$ and the lower bound
from the fact that $v_{d+k} = d$ and $v_{s-1} \ge v_s-1$. We have shown

\begin{prop}
  The nonzero $(d+k)\times (d+k)$ minors of $M_k$ correspond exactly to minor walks
  of length $d+k$. The leading monomial of the minor corresponding to a walk is
  obtained by multiplying the variables corresponding to the visited lattice
  points, counted with multiplicity.
\end{prop}
%More intuitively, but less formally, a minor walk is obtained by choosing a nonzero minor of $M_k$, labelling the diagnal entries $1,\dots,d+k$ and forgetting all the other entries.
%\begin{remark}
%  There is an involution on minor walks: If $(u_1,v_1),
%  \ldots, (u_{d+k},v_{d+k})$ is a minor walk, then so is
%  $ (n+1-u_{d+k},d-v_{d+k}), \ldots, (n+1-u_1,d-v_1)$.
%\end{remark}
If any subset of the minors of Theorem \ref{thm:minors} forms a Gr\"obner
basis, then so must a subset whose leading terms divide the leading terms of any
minor. Let us construct a minimal such subset, which we will see actually
corresponds to a unique set of minors. Let a \emph{reduced} minor walk denote a
minor walk which is minimal under inclusion, i.e., one for which no vertex can
be deleted and remain a minor walk.
\begin{lemma}\label{lem:reduced}
  A minor walk $(u_1,v_1), \ldots, (u_{d+k},v_{d+k})$ is reduced if and only if
  \begin{enumerate}
    \item $v_{s+1} \ge v_s$;
    \item $v_2 = 1$ and $v_{d+k-1} = d-1$;
    \item if $v_{s+1} = v_s$, then $v_s = v_{s-1}+1$, $v_{s+2} = v_{s+1} +
      1$, $u_s \ge u_{s+2}$ and $u_{s-1} \geq u_{s+1}$. In particular, $u_{s} > u_{s+1}$.
  \end{enumerate}
  Furthermore, a reduced minor walk visits each vertex at most once, is
  determined by the set of vertices it visits, and no minor walk can visit a
  proper subset of the visited vertices.
\end{lemma}
\begin{proof}
  Let $(u_1,v_1),\ldots,(u_{d+k},v_{d+k})$ be a minor walk. We first show if any of the conditions of the claim are violated, this walk is not reduced.
  \begin{enumerate}
      \item If $v_{s+1}<v_s$, then $v_{s+1}\leq v_{s-1}$ and thus we can remove the $s$-th step.
      \item If $v_2\neq 1$, then $v_2=0$ and we may remove the first step. Analogously if $v_{d+k-1}\neq d-1$ we may remove the last step.
      \item Suppose $v_{s+1}=v_s$. If $v_s\neq v_{s-1}+1$ or $u_{s-1} < u_{s+1}$,
        we may remove the $s$-th step. 
        If $v_{s+2}\neq v_{s+1}+1$ or $u_s < u_{s+2}$ we may remove the $(s+1)$-st step.
  \end{enumerate}
  Conversely, suppose the walk satisfies the conditions, and consider
  $(u_s,v_s)$ and $(u_t,v_t)$ for $t\ge s + 2$. We have either $v_t\ge v_s + 2$
  or both $v_t = v_s+1$ and $u_s \geq u_t$. In either case, it is illegal to
  make such a step directly in a minor walk, and it follows that the walk is reduced. 

  Now, note that a reduced walk can visit a vertex at most once, as otherwise
  we could remove the part of the walk from leaving a given vertex until coming
  back to it.
  For the remaining claims, consider any minor walk visiting a subset of $\{(u_1,v_1),\ldots,
  (u_{d+k},v_{d+k})\}$.
  Such a walk must begin with $(u_1,v_1)$ and end with
  $(u_{d+k},v_{d+k})$, so for each $s$, it must at some point pass from the set
  $\{(u_1,v_1),\ldots,(u_s,v_s)\}$ to
  $\{(u_{s+1},v_{s+1},\ldots,(u_{d+k},v_{d+k})\}$. By the reasoning of the last
  paragraph, the only legal step accomplishing this is $(u_s,v_s),
  (u_{s+1},v_{s+1})$. This establishes the remaining claims.
\end{proof}

Let $G$ be the set of minors corresponding to reduced walks. From the second
claim of Lemma~\ref{lem:reduced}, every leading term not divisible by another is represented in
$G$ exactly once, and no others are. From the definition of reduced, it is clear
that the leading term of any minor is divisible by that of one in $G$.
Furthermore, by the same Lemma, it is clear that reduced walks
must have length at
most $2d$ and that this length is achievable when $n\ge 2$. Walks of length at
most $2d$ correspond to minors of $M_k$, $k\le d$, which accounts for the reason
the claim of Theorem \ref{thm:minors} is as it is and is sharp.

Now let us determine $V(\lt G)$, which consists of coordinate subspaces. The
equations of one of the components consist of a minimal subset of variables so
that all generators of $\lt G$ are divisible by at least one in the subset.

\begin{prop}\label{prop:V(ltG)}
  Let $1\le s\le n$ and $1\le t \le d$. Write 
  \[S_{s,t} = \{ a_{i,t-1}, i < s
  \} \cup \{ a_{i,t} , i > s \}.\] 
  Then $S_{s,t}$ is an inclusion minimal subset
  of variables intersecting the vertices of every (reduced) minor walk, and all
  such subsets are one of the $S_{s,t}$. 

  Thus, there are $nd$ such subsets, each of size $n-1$. In particular $V(\lt
  G)$ is equidimensional of projective dimension $nd$ and degree $nd$.
\end{prop}
\begin{proof}
  A minor walk must intersect $S_{s,t}$ at the beginning or end of any step it
  passes from column $t-1$ to $t$, and it must do this at least once. It is easy
  to construct a minor walk avoiding any proper subset of $S_{s,t}$, so the
  first claim is shown.

  Now, let $S$ be an inclusion minimal subset intersecting any walk, and
  identify variables with corresponding lattice points. Suppose $(i_1,j_1),
  (i_2,j_2) \in S$, $i_1\le i_2$. Since $S$ is minimal under inclusions, there
  must be a minor walk avoiding any proper subset of $S$, in particular there is
  a pair of minor walks that avoid $S$ except for exactly
  $(i_1,j_1)$ and $(i_2,j_2)$, respectively.
  
  Suppose either $i_1 +2 \le i_2$ or $i_1 + 1 = i_2$ and $j_1\ge j_2-1$. Then
  the prefix of the path into and excluding the first occurrence of $(i_2,j_2)$
  followed by the suffix out of and excluding the last occurrence $(i_1,j_1)$ is
  a minor walk avoiding $S$, contradiction. Hence $S$ is confined to one column, in which case it equals $S_{1,t}$ or $S_{n,t}$, or to two
  adjacent columns, and any element of $S$ in the left column rules out any in
  the right column above one below the element. 
  Such a set $S$ is then a subset
  of an $S_{s,t}$, and by minimality of $S_{s,t}$ equal.
\end{proof}
We can now finish the proof of our main result.
%MM: below we slighlty repeat ourselves, but maybe it better to have it as a formal proof
\begin{proof}[Proof of Theorem \ref{thm:minors}]
    We will prove that $G$ is a Gr\"obner basis of the resultant ideal $I_{d,n}$, that is, $\lt
  G$ generates the initial ideal of $I_{d,n}$. By Proposition \ref{prop:V(ltG)}, we know that $V(\lt G)$ is a reduced, equidimensional variety of degree $dn$ and dimension $dn$. For any ideal $J$ that strictly contains $\lt G$, the variety $V(J)$ must be strictly included in $V(\lt G)$. In particular, it must have either strictly smaller dimension, or the same dimension and strictly smaller degree. 
  However, $V(\lt I_{d,n})$ has the same dimension and degree as $V(I_{d,n})$, that is, $dn$ and $dn$ by Corollary \ref{cor:degdimdn}. Thus, $\lt I_{d,n}$ cannot strictly contain the ideal generated by $\lt G$, and as $G\subset I_{d,n}$, $G$ is a Gr\"obner basis for $I_{d,n}$ and in particular generates it.
\end{proof}

\begin{remark} A question arises as to whether one could explicitly identify a {\em minimal} generating set for $I_{d,n}$, perhaps within our Gr\"obner basis $G$. Our work sheds no light on this interesting question. The problem of finding a small set of equations for the resultant locus, albeit in the set-theoretic sense, was studied in~\cite{Lub}. 
\end{remark}

\section{Final remarks}\label{sect_conc}
We have provided a complete set of generators for the resultant ideal $I_{d,n}$. 
% It is even prime, as X_d,n is naturally parameterized as the image of an irreducible variety
We proceed to explain how this is related to equations defining $X_{d,n}=V(I_{d,n}) \subset \PP^{n(d+1)-1}$.
\begin{lemma}
    If the rank of $M_k$ is strictly smaller than $d+k$, then the rank of $M_{k-1}$ is strictly smaller than $d+k-1$.
\end{lemma}
\begin{proof}
    Suppose for contradiction that $M_{k-1}$ has rank $d+k-1$. We may consider $M_{k-1}$ as an upper left submatrix of $M_k$. Thus the rank of $M_k$ and $M_{k-1}$ would have to be equal. This would be only possible if the last column of $M_k$ is zero. But in this case, so would be the last row of $M_{k-1}$, which is a contradiction. 
\end{proof}
We recover the set theoretic result by Orsinger and Kaki\'e.
\begin{corollary} \label{cor_kakie} The set-theoretic zero locus of all $2d\times 2d$ minors of $M_d$ is the variety $X_{d,n}\subset \PP^{n(d+1)-1}$. 
\end{corollary}
As noted before, in general these minors clearly cannot generate the ideal $I_{d,n}$, as smaller minors have smaller degree. However, the following result holds\footnote{This statement and its the proof were communicated to us by Jan Stevens after we posted the first version of the article on the arXiv.}. 
%As we show however in the next proposition, the following result is true. In fact the $2d\times 2d$ minors of $M_d$ define an ideal whose saturation with respect to the irrelevant ideal is $I_{d,n}$. 
\begin{prop}\label{prop:sat}
Let $J_{d,n}$ be the ideal generated by all $2d\times 2d$ minors of $M_d$, and let $m_{d,n}$ be the irrelevant ideal in $K[a_{i,j}]$. The saturation of $J_{d,n}$ with respect to $m_{d,n}$ equals $I_{d,n}$. Equivalently, $J_{d,n}$ and $I_{d,n}$ define the same projective scheme.
\end{prop}
\begin{proof}
    The equivalence of the two claimed statements is well known~\cite[Ex.II.5.10(b)]{Hartshorne}. We prove that the two projective schemes are equal. For this we need to prove that for any point in the projective space there is an affine neighbourhood on which the two schemes are equal. 

    First we note that the group $GL(2)$ acts on the projective space $\PP^{n(d+1)-1}$ as the change of variables $x,y$. This action clearly preserves $I_{d,n}$, as we know that this is the prime ideal of the locus when the forms have a common root, and this condition does not depend on the choice of coordinates.
    In fact, $GL(2)$ also acts on $J_{d,n}$, which can be seen through the intrinsic description of $J_{d,n}$ as a Fitting ideal. %Alternatively, one may simply notice that the action of $GL(2)$ is simply linear action on the columns of $M_k$ for any $k$ (in particular, $M_d$) and thus does not change the ideal of maximal minors. 

    Pick any point $p\in \PP^{n(d+1)-1}$ corresponding to an $n$-tuple of degree $d$ polynomials. One of those polynomials must be nonzero and without loss of generality we assume it is the first one. Note that in this case, we may act with an element of $GL(2)$ so that $a_{1,0}\neq 0$. Thus to compare the projective schemes defined by $J_{d,n}$ and $I_{d,n}$ it is enough to compare them on the affine chart $a_{1,0}=1$. Thus we have to prove that the two ideals are equal after we substitute $a_{1,0}=1$. By definition, $J_{d,n}\subset I_{d,n}$. Pick any generator of $I_{d,n}$, that is a maximal minor of some matrix $M_k$. Note that $M_k$ may be realized as a submatrix $M_k'$ of $M_d$ in the last $(d+k)$ columns and last $nk$ rows. Adding rows $1,n+1, 2n+1,\dots, (d-k-1)n+1$ to those of $M_k'$ and considering all columns, we obtain a submatrix of $M_d$ with maximal minors equal to maximal minors of $M_k$ after we substitute $a_{1,0}=1$. Indeed, the chosen submatrix on first $(d-k)$ columns is upper triangular with $a_{1,0}$ on its diagonal.

    This shows that the two ideals are equal after we substitute $a_{1,0}=1$, and thus finishes the proof of the proposition.
\end{proof}

We finally note that one of the main steps of the proof was finding a square-free Gr\"obner basis of $I_{d,n}$. There exist other term orders that provide square-free initial ideals, which however do not choose the diagonal as the leading term.
\begin{example} Consider the first non-trivial case $d=2, n=3$. We calculate minors and leading terms in
\texttt{degrevlex} polynomial ordering using Macaulay2~\cite{M2}. 
\begin{verbatim}
R = QQ[a_1..a_3,b_1..b_3,c_1..c_3];
N = matrix {
{a_1,b_1,c_1,0},
{a_2,b_2,c_2,0},
{a_3,b_3,c_3,0},
{0,a_1,b_1,c_1},
{0,a_2,b_2,c_2},
{0,a_3,b_3,c_3}};

M = matrix {
{a_1,b_1,c_1},
{a_2,b_2,c_2},
{a_3,b_3,c_3}};

U = (minors(3,M) + minors(4,N))
LU = leadTerm U
\end{verbatim}

The output is
\[a_3b_2c_1, a_3b_2b_3c_2, a_3b_1b_3c_2, a_2b_1b_3c_2, a_3b_1b_3c_1, a_2b_1b_3c_1, a_2b_1b_2c_1\]
%Elimination gives the same ideal (the output is \texttt{true}):
%\begin{verbatim}
%I = eliminate(x, ideal(a_1*x^2 + b_1*x + c_1,
%                       a_2*x^2 + b_2*x + c_2,
%                       a_3*x^2 + b_3*x + c_3));
%I == U
%\end{verbatim}
See~\cite{tropical}*{Ex.\;6.6} for a different analysis of this example. 
\end{example}

\end{document}